\documentclass[11pt,a4paper,reqno]{amsart}
\usepackage{mystyle}

\author[Hladk\'y]{Jan Hladk\'y}
\address[Hladk\'y]{Institute of Computer Science of the Czech Academy of Sciences, Pod Vod\'arenskou v\v{e}\v{z}\'i 2, 182 07 Prague, Czechia.}
\email{hladky@cs.cas.cz}
\thanks{(Hladk\'y): Research supported by Czech Science Foundation Project 21-21762X}

\author[I\v{l}kovi\v{c}]{Daniel I\v{l}kovi\v{c}}
\address[I\v{l}kovi\v{c}, Shu]{Institute of Mathematics, Leipzig University, Augustusplatz 10, 04109 Leipzig.\newline \hspace*{3mm} Previous affiliation: Faculty of Informatics, Masaryk University, Botanick\'a 68A, 602 00 Brno, Czech Republic.}
\email{daniel.ilkovic@gmail.com}
\thanks{(I\v{l}kovi\v{c}, Shu): Research supported by the Alexander von Humboldt Foundation through the
Alexander von Humboldt Professorship of Daniel Kr{\'a}l', endowed by the
Federal Ministry of Education and Research, and by the Grant Agency of
Masaryk University through the MUNI Award in Science and Humanities
(MUNI/I/1677/2018).}

\author[Le\'on]{Jared Le\'on}
\address[Le\'on]{Mathematics Institute, University of Warwick,
  Coventry CV4 7AL, United Kingdom}
\email{jared.leon@warwick.ac.uk}
\thanks{(Le\'on): Research supported by the Warwick Mathematics Institute Centre for Doctoral Training.}

\author[Shu]{Xichao Shu}
\email{xichao.shu@uni-leipzig.de}

\begin{document}

\title{Arithmetic progressions in a random set on a budget}

\begin{abstract}
%ORIGINAL ABSTRACT:
%We study the problem of constructing $k$-term arithmetic progressions in a random subset of integers under a restricted budget. A player, Builder, is presented with a random sequence of $t$ integers drawn from $[n]$ and must immediately and irrevocably decide whether to select each integer, subject to a maximum budget $b$. We establish the optimal thresholds for this process, proving that for $t = \omega(n^{1-2/k})$, a budget of $b = \Theta((n/t)^{\frac{k-2}{2}})$ is both necessary and sufficient for Builder to successfully construct a $k$-term arithmetic progression with high probability.
%---------------------------------
%HONZA VERSION 1, just made clearer that things are done in a step-by-step fashion
%We study the problem of constructing $k$-term arithmetic progressions in a random subset of integers $\{1,\ldots,n\}$ under a restricted budget. A player, Builder, is presented with a sequence of $t$ integers drawn uniformly at random from $\{1,\ldots,n\}$. As the elements are revealed one by one, Builder must immediately and irrevocably decide whether to select the current integer, subject to a maximum budget of $b$ selected elements in total. We establish the optimal thresholds for this process, proving that for $t = \omega(n^{1-2/k})$, a budget of $b = \Theta((n/t)^{\frac{k-2}{2}})$ is both necessary and sufficient for Builder to successfully construct a $k$-term arithmetic progression with high probability.
%---------------------------------
%VERSION 2, HONZA'SPREFERED: connects back to the graph framework
A restricted-budget version of the random graph process, introduced by Frieze, Krivelevich, and Michaeli in 2025, studies the construction of structures by an online player who can purchase only a limited number of random edges. In this paper, we transfer this framework from random graphs to random subsets of integers, focusing on the construction of $k$-term arithmetic progressions. A player, Builder, is presented with a sequence of $t$ integers drawn uniformly at random from $[n]$. As the elements are revealed one by one, Builder must immediately and irrevocably decide whether to select the current integer, subject to a maximum budget of $b$ selected elements in total. We establish the optimal thresholds for this process, proving that for $t = \omega(n^{1-2/k})$, a budget of $b = \Theta((n/t)^{\frac{k-2}{2}})$ is both necessary and sufficient for Builder to successfully construct a $k$-term arithmetic progression with high probability.
\end{abstract}

\maketitle

\section{Introduction}
Much of the theory of random discrete structures concerns the emergence
of prescribed substructures. In this paper, we focus on two random
models: Erd\H{o}s--R\'enyi random graphs and random subsets of $[n]$.
The former provides the main source of inspiration, whereas the latter
is the object of study.

The uniform random graph $G(n,m)$, where $0\leq m\leq \binom{n}{2}$,
may be generated by ordering the edges of $K_n$ uniformly at random
and retaining the first $m$ edges.  Equivalently, $G(n,m)$ is the
graph at time $m$ in the standard random graph process
\cite{MR0125031}. Classical results show that connectivity---or,
equivalently, the appearance of a spanning tree---and, when $n$ is
even, the appearance of a perfect matching both have their critical
windows centred at $\frac{1}{2}n\log n$
\cite{MR0125031,ErdosRenyi1966,BollobasThomason1985}.  Hamiltonicity
occurs slightly later, with its critical window centred at
$\frac{1}{2}n(\log n+\log\log n)$
\cite{Bollobas1984,AjtaiKomlosSzemeredi1985}.  Thus, all three
properties have the same first-order threshold
\begin{equation*}
  m=\left(\frac{1}{2}+o(1)\right)n\log n.
\end{equation*}
For general background on random graphs, we refer the reader to
\cite{Bollobas2001,MR1782847,FriezeKaronski:IntroRandomGraphs}.

Frieze, Krivelevich, and Michaeli~\cite{FriezeKrivelevichMichaeli} enhanced the standard random graph process with the notion of a restricted budget. In their framework, the edges of $K_n$ are presented one at a time in a uniformly random order. A player, known as Builder, must make an immediate and irrevocable decision whether to purchase each presented edge. The core challenge is that Builder operates under two strict limits: a time constraint on the number of edges they can inspect, and a budget constraint on the number they can purchase. The ultimate goal is to construct a prescribed structure, such as a spanning tree, a perfect matching, or a Hamilton cycle.

Several subsequent works have investigated how small the budget can be while retaining the usual random-graph threshold for the duration of the process. For every fixed $d\in\mathbb N$, Lichev~\cite{Lichev} gave an asymptotically optimal strategy for constructing a $d$-connected graph.
Anastos~\cite{Anastos} obtained asymptotically optimal results for perfect matchings and Hamilton cycles. Further hitting-time and minimum-degree variants were studied by Frieze, Krivelevich, and Michaeli~\cite{FriezeKrivelevichMichaeli} and by Katsamaktsis and Letzter~\cite{KatsamaktsisLetzter}. These results show that, for several classical random-graph properties, Builder can retain the desired structure while purchasing only a small proportion of the edges that have been exposed. 
%HONZA: Work of Espuny Dias et al has been added.
Espuny Díaz, Garbe, Naia, and
Smith~\cite{diaz2026graphfactorspowershamilton} studied partial
$F$-factors and powers of Hamilton cycles, showing that for these
structures, the required budget cannot always be asymptotically
reduced.

The problem of constructing small subgraphs has also been
considered. Frieze, Krivelevich, and Michaeli found the optimal budget
required to construct any fixed tree or cycle, as well as any
unicyclic graph \cite{FriezeKrivelevichMichaeli}. The optimal budget
of several other multicyclic graphs, such as \(K_4^-\), \(K_4\), fans and
wheels were also determined
\cite{IlkovicLeonShu24,AntoniukDiazPetrovaStojakovic26}. For general
graphs, however---and in particular, cliques---the problem remains
wide open.

In this short paper, we transfer the framework of processes with restricted budget from graphs to integers. The particular structures that we study are arithmetic progressions of a fixed length. A non-trivial \emph{$k$-term arithmetic progression}, abbreviated to a \emph{$k$-AP}, is a set of the form $\{a,a+d,\ldots,a+(k-1)d\}$, where $d\geq 1$. Arithmetic progressions are among the central objects of additive combinatorics. Roth proved that every subset of the positive
integers with positive upper density contains a non-trivial $3$-term arithmetic progression~\cite{Roth1953}, and Szemer\'edi extended this to arithmetic progressions of every fixed length~\cite{Szemeredi1975}.

%Sparse random analogues of these density theorems ask whether every subset of positive relative density inside a random subset of $[n]$ contains a prescribed arithmetic progression. The case of $3$-term arithmetic progressions was established by Kohayakawa, \L{}uczak, and R\"odl~\cite{KohayakawaLuczakRodl1996}. The corresponding result for arithmetic progressions of every fixed length was later obtained independently by Conlon and Gowers~\cite{ConlonGowers2016} and by Schacht~\cite{Schacht2016}.
%HONZA: I MADE SUBSTANTIAL CHANGES HERE.
Sparse random analogues of these density theorems ask whether every subset of positive relative density inside a random subset of $[n]$ contains a $k$-AP. It is a standard observation that this property requires a density of at least $p = \Omega(n^{-1/(k-1)})$. The optimal result for $3$-APs at this density was established by Kohayakawa, \L{}uczak, and R\"odl~\cite{KohayakawaLuczakRodl1996}, and the corresponding result for every fixed length was later obtained independently by Conlon and Gowers~\cite{ConlonGowers2016} and by Schacht~\cite{Schacht2016}.

Our problem is different from these random versions of Szemer\'edi's theorem. We do not ask whether every sufficiently large subset of a random set contains an arithmetic progression. Instead, we consider an online selection problem: the elements of $[n]$ are presented one by one in a uniformly random order, but the process terminates after only $t$ steps. Builder must select a small number of them, without knowledge of the future, so that the selected set contains a single $k$-AP before the sequence ends.

\begin{definition}
\label{def:strategy}
Let $0\leq b\leq t\leq n$. Let $X_1,\ldots,X_n$ be a uniformly random permutation of $[n]$. A \emph{$(t,b)$-strategy} is a possibly randomised online algorithm for a player called Builder.
During round $s\in[t]$, Builder is shown $X_s$ and must decide immediately and irrevocably whether to purchase it. This decision may depend only on the integers $X_1,\ldots,X_s$, the previous purchase decisions, and the internal randomness of the strategy. Builder may purchase at most $b$ integers in total.
\end{definition}

A $(t,b)$-strategy \emph{succeeds} if the set of purchased integers
contains a $k$-AP. Throughout the
paper, $k\geq 3$ is fixed and $n\to\infty$. As usual, an event holds
\emph{with high probability}, abbreviated as \whp, if its
probability tends to $1$ as $n\to\infty$.

Our main result determines, in the usual $o/\omega$ sense, the joint
threshold for the duration of the process and the purchasing budget.

\begin{theorem} \label{thm:main}
Let $k\ge3$ be fixed. Let $n\ge1$, and let $t, b\le n$ be functions of
$n$. Then, if $t=\omega(n^{1-2/k})$ and
$b=\omega((n/t)^{\frac{k-2}{2}})$ there is a $(t, b)$-strategy that \whp constructs a $k$-AP. Conversely, if $t=o(n^{1-2/k})$ or $b=o((n/t)^{\frac{k-2}{2}})$ then any $(t, b)$-strategy \whp fails to construct a $k$-AP.
\end{theorem}

%HONZA: MORE DETAILS ADDED
The two conditions in Theorem~\ref{thm:main} correspond to two natural obstructions. The first is the usual threshold $t\asymp n^{1-2/k}$ for the appearance of a single $k$-AP among the exposed integers.\footnote{We recall this result in Proposition~\ref{prop:threshold}.} The second is the budget obstruction: the $b$ purchased elements can form at most $O(b^2)$ pairs, and for any such pair to be extended into a $k$-AP, the remaining $k-2$ elements of the progression must appear among the $t$ exposed integers. This occurs with probability roughly $(t/n)^{k-2}$, meaning the expected number of constructed progressions is bounded by $O(b^2(t/n)^{k-2})$. Requiring this expectation to be non-vanishing yields the critical scale $b\asymp(n/t)^{\frac{k-2}{2}}$.

The positive part is proved by purchasing exposed integers from a
suitable initial interval of $[n]$. The converse follows from a
first-moment argument and a count of partial arithmetic-progression
templates.

\section{Notation and Preliminaries}

For a positive integer $n$ and $p\in[0,1]$, we write $[n]_p$ for the
binomial random subset of $[n]$ obtained by including every element
independently with probability $p$. For an integer $0\leq m\leq n$, we
write $[n]_m$ for a uniformly random $m$-element subset of $[n]$.

A property (i.e., a family) $\mathcal P$ of subsets of $[n]$ is called \emph{increasing} if for each $A\in\mathcal P$ and $A\subseteq B\subseteq[n]$ we have $B\in\mathcal P$. Whenever an asymptotic expression is used to define an integer parameter, we round it to an integer in an arbitrary way. This does not affect any of our arguments.

We shall use the following standard concentration inequalities.
\begin{lemma}[Chebyshev's inequality]
\label{lem:chebyshev}
Let $X$ be a random variable with finite variance. Then, for every
$a>0$,
\[
    \mathbb P\bigl(|X-\mathbb EX|\geq a\bigr)
    \leq
    \frac{\operatorname{Var}(X)}{a^2}.
\]
\end{lemma}

\begin{lemma}[Chernoff's inequality]
\label{lem:chernoff}
Let $X$ be a binomial random variable with mean $\mu$. Then, for every
$0<\varepsilon\leq1$,
\[
    \mathbb P\bigl(|X-\mu|\geq\varepsilon\mu\bigr)
    \leq
    2\exp\left(-\frac{\varepsilon^2\mu}{3}\right).
\]
\end{lemma}
%HONZA: Previously, McDiarmid was used but not spelled out. Now added the statement.
\begin{lemma}[McDiarmid's inequality]
\label{lem:mcdiarmid}
Let $X_1, \dots, X_T$ be independent random variables taking values in a set $\mathcal{X}$. Let $f \colon \mathcal{X}^T \to \mathbb{R}$ be a function satisfying the bounded differences condition
\[
|f(x) - f(x')| \le 1
\]
whenever $x, x' \in \mathcal{X}^T$ differ in exactly~1 coordinate. Then, for every $\varepsilon > 0$,
\[
\mathbb{P}\bigl(f(X_1, \dots, X_T) - \mathbb{E}[f(X_1, \dots, X_T)] \le -\varepsilon\bigr) \le \exp\left(-\frac{2\varepsilon^2}{T}\right).
\]
\end{lemma}

We shall also use the standard expectation and variance formulas for the
hypergeometric distribution. If $H$ counts the number of marked elements
obtained by choosing $m$ elements uniformly without replacement from a
set of size $n$ containing $K$ marked elements, then
$\mathbb EH=\frac{mK}{n}$ and
\[
    \operatorname{Var}(H)
    =
    m\frac{K}{n}
    \left(1-\frac{K}{n}\right)
    \frac{n-m}{n-1}
    \leq \mathbb EH.
\]

We next record the threshold for the appearance of a single arithmetic progression in a binomial random subset. For completeness, we include a proof.

\begin{proposition}
\label{prop:threshold}
Let $k\geq3$ be fixed, and let $p=p(n)\in[0,1]$. If
$p=\omega\left(n^{-2/k}\right)$, then \whp the random set $[n]_p$
contains a $k$-AP. If
$p=o\left(n^{-2/k}\right)$, then \whp the random set $[n]_p$ contains no $k$-AP.
\end{proposition}

\begin{proof}
Let $\mathcal A_k(n)$ be the family of all non-trivial $k$-term arithmetic progressions in $[n]$, and let $X$ count those contained in $[n]_p$.
Since $|\mathcal A_k(n)|=\Theta_k(n^2)$, we have
$\mu:=\mathbb EX=\Theta_k(n^2p^k)$. If $p=o(n^{-2/k})$, then
$\mu=o(1)$, and Markov's inequality gives $\mathbb P(X>0)=o(1)$.

Now suppose that $p=\omega(n^{-2/k})$. There are $O_k(n^3)$ ordered
pairs of distinct $k$-APs intersecting in exactly one element, and
$O_k(n^2)$ ordered pairs intersecting in at least two elements.
Moreover, the union of two distinct $k$-APs has size at least $k+1$.
Consequently,
\[
    \operatorname{Var}(X)
    \leq
    \mu
    +O_k(n^3p^{2k-1})
    +O_k(n^2p^{k+1}).
\]
Since $\mu=\Theta_k(n^2p^k)$,
\[
    \frac{\operatorname{Var}(X)}{\mu^2}
    =
    O_k\left(
        \frac{1}{n^2p^k}
        +\frac{1}{np}
        +\frac{1}{n^2p^{k-1}}
    \right)
    =o(1).
\]
Thus, by Chebyshev's inequality,
\[
    \mathbb P(X=0)
    \leq
    \frac{\operatorname{Var}(X)}{\mu^2}
    =o(1).
\]
\end{proof}

We shall also need a standard comparison between the binomial and
uniform random-set models.

\begin{lemma}
\label{lem:coupling}
Let $\mathcal P$ be an increasing property of subsets of $[n]$, and let
$p=p(n)$ satisfy $pn\longrightarrow\infty$. Let $m=m(n)$ be an integer
with $0\leq m\leq n$. Then the following statements hold.
\begin{enumerate}
\item[(a)] If $m\geq2pn$, then
$\mathbb P\bigl([n]_m\in\mathcal P\bigr)
\geq\mathbb P\bigl([n]_p\in\mathcal P\bigr)-o(1)$.

\item[(b)] If $m\leq pn/2$, then
$\mathbb P\bigl([n]_m\in\mathcal P\bigr)
\leq\mathbb P\bigl([n]_p\in\mathcal P\bigr)+o(1)$.
\end{enumerate}
\end{lemma}

\begin{proof}
Let $\pi$ be a uniformly random permutation of $[n]$, and let
$Y\sim\operatorname{Bin}(n,p)$ be independent of $\pi$. Define
$A=\{\pi(1),\ldots,\pi(Y)\}$ and
$B=\{\pi(1),\ldots,\pi(m)\}$. Then $A\sim[n]_p$ and $B\sim[n]_m$.

Suppose first that $m\geq2pn$. On the event $Y\leq m$, we have
$A\subseteq B$. Since $\mathcal P$ is increasing,
\[
    \mathbb P(A\in\mathcal P)
    \leq
    \mathbb P(B\in\mathcal P)+\mathbb P(Y>m).
\]
By Chernoff's inequality,
\[
    \mathbb P(Y>m)
    \leq
    \mathbb P(Y>2pn)
    =
    o(1),
\]
because $pn\to\infty$. This proves Part~\textup{(a)}.

Suppose now that $m\leq pn/2$. On the event $Y\geq m$, we have
$B\subseteq A$. Hence,
\[
    \mathbb P(B\in\mathcal P)
    \leq
    \mathbb P(A\in\mathcal P)+\mathbb P(Y<m).
\]
Again, Chernoff's inequality gives
\[
    \mathbb P(Y<m)
    \leq
    \mathbb P(Y<pn/2)
    =
    o(1),
\]
which proves Part~\textup{(b)}.
\end{proof}

\section{Proof of Theorem \ref{thm:main}}

\subsection*{Proof of the 1-statement.} Let $k\geq3$, $t=\omega\left(n^{1-2/k}\right)$ and $b=\omega\left(\left(\frac{n}{t}\right)^{(k-2)/2}\right)$.
Suppose that $m=\left\lfloor\frac{bn}{2t}\right\rfloor$. The strategy for the 1-statement purchases every exposed element
belonging to $[m]$, provided that fewer than $b$ elements have been
purchased so far. We shall show that \whp the set
$[n]_{t}\cap[m]$ has size at most $b$ and contains a $k$-AP. On this
event, the strategy purchases every element of $[n]_{t}\cap[m]$ and
therefore succeeds.

We first show that with high probability such a set contains a $k$-AP. Let $p=\frac{t}{2n}$ and let $A$ be the random set $[n]_{p}\cap[m]$. Clearly $A\sim[m]_{p}$. Hence, $\mathbb{E}|A|=pm=b/4$. By Chernoff's inequality (Lemma \ref{lem:chernoff}), we have
\[ \mathbb{P}(|A|\ge b)\le\mathbb{P}(|A|-\mathbb{E}|A|\ge\mathbb{E}|A|)=2e^{-b/12}=o(1). \]
We can also see that $p=\omega(m^{-2/k})$. Indeed, given that $b^{2}p^{k-2}=b^{2}(t/2n)^{k-2}=\omega(1)$,
\[ 
\frac{p}{m^{-2/k}} = p m^{2/k} = p \left(\frac{b}{4p}\right)^{2/k} = \left(\frac{b^{2}p^{k-2}}{16}\right)^{1/k} = \omega(1).
\]
Then by Proposition \ref{prop:threshold}, \whp $A$ contains a $k$-AP.

Hence, \whp the random set $[n]_{p}$ satisfies that the size of
$A=[n]_{p}\cap[m]$ is less than $b$, and furthermore, $A$ contains a
$k$-AP. The property of containing a $k$-AP is clearly increasing, and
therefore, by Lemma~\ref{lem:coupling}(a), the random set $[n]_{t}$
also satisfies it \whp.

We now show that we do not exceed the budget when constructing this set. In order to show this, let $B$ be the random set $[n]_{t}\cap[m]$, and observe that the random variable $|B|$ follows a hypergeometric distribution with parameters $t$, $n$ and $m$, that is, it counts the number of special objects when we select $t$ elements without replacement from a population of size $n$ that contains $m$ special objects. We then have
\[ \mathbb{E}|B|=\frac{tm}{n} \quad \text{and} \quad \text{Var}|B|=\frac{tm}{n}\left(\frac{n-m}{n}\right)\left(\frac{n-t}{n-1}\right)\le\mathbb{E}|B| \]
Notice also that $\mathbb{E}|B|=tm/n=b/2=\omega(1)$. Therefore, by Chebyshev's inequality (Lemma~\ref{lem:chebyshev}),
\[ \mathbb{P}(|B|\ge b)=\mathbb{P}(|B|-\mathbb{E}|B|\ge\mathbb{E}|B|)\le\frac{1}{\mathbb{E}|B|}=o(1). \]
By the preceding two estimates, \whp the set
$B=[n]_{t}\cap[m]$ has size at most $b$ and contains a $k$-AP.
Consequently, the strategy purchases every element of $B$ and succeeds.
This completes the proof of the 1-statement.

\subsection*{Proof of the 0-statement.}
First suppose that $t=o(n^{1-2/k})$. Since there are $O_k(n^2)$
non-trivial $k$-APs in $[n]$, the expected
number of such progressions contained in $[n]_t$ is at most
\[
    O_k(n^2)\frac{(t)_k}{(n)_k}
    \leq O_k\left(n^2\left(\frac{t}{n}\right)^k\right)
    =
    O_k\left(\frac{t^k}{n^{k-2}}\right)
    =
    o(1).
\]
Thus, by Markov's inequality, \whp the set $[n]_t$ contains no $k$-AP.

It remains to consider the case
$b=o((n/t)^{\frac{k-2}{2}})$. By the subsequence principle, it is enough
to consider a subsequence on which either
$t=o(n^{1-2/k})$ or $t=\Omega(n^{1-2/k})$. The former case has already
been handled, so we may assume that $t=\Omega(n^{1-2/k})$, and in
particular $t\to\infty$. If $b=0$, the conclusion is immediate.
Otherwise, after passing to a further subsequence, we may assume that
$b\geq1$. The assumption $b=o((n/t)^{\frac{k-2}{2}})$ then implies that
$t=o(n)$.

Let $T=2t$. Suppose that $X_{1},\dots,X_{T}$ are i.i.d.\ random
variables, where each $X_i$ is taken uniformly at random from $[n]$.
Let $Y$ be the number of distinct values observed among these $T$
variables. We have
\[
    \mathbb{E}Y
    =
    n\left(1-\left(1-\frac1n\right)^{2t}\right)
    \geq
    2t-\frac{\binom{2t}{2}}{n}
    \geq
    2t\left(1-\frac{t}{n}\right).
\]
Since $t=o(n)$, we have $\mathbb EY\geq 3t/2$ for all sufficiently large $n$. Moreover, changing one of the variables $X_i$ changes $Y$ by at most $1$. Hence, by McDiarmid's inequality (Lemma~\ref{lem:mcdiarmid}),
\[
    \mathbb{P}(Y<t)
    \leq
    \exp\left(-\frac{2(\mathbb EY-t)^2}{T}\right)
    \leq
    \exp(-t/4)
    =
    o(1).
\]

We now justify the reduction to the iid model. Let $\mathcal S$ be an
arbitrary $(t,b)$-strategy in the original without-replacement model.
Define a $(T,b)$-strategy $\mathcal S'$ for the iid model by ignoring
repeated values and applying $\mathcal S$ to the first $t$ distinct
values that appear. Conditional on $Y\geq t$, these first $t$ distinct
values, in their order of appearance, have the same distribution as a
uniformly random ordered $t$-element subset of $[n]$. Therefore,
\[
    \mathbb P(\mathcal S\text{ succeeds})
    \leq
    \mathbb P(\mathcal S'\text{ succeeds})+\mathbb P(Y<t).
\]
Consequently, it is enough to show that every $(T,b)$-strategy presented
with $(X_i)_{i\in[T]}$ succeeds with probability $o(1)$.

A \emph{template} is a $k$-AP contained in $[n]$. With a slight abuse of terminology, we say that a set $B\subset[n]$ \emph{contains an access-$i$ template $\Lambda$} if at least $i$ elements of $\Lambda$ belong to $B$. We also say that a template $\Lambda$ \emph{contains} an element $x\in[n]$ if $x\in \Lambda$. Notice that $B$ contains a $k$-AP if and only if it contains an access-$k$ template.

Fix a $(T,b)$-strategy, and let $B\subset[n]$ be the set of distinct
values purchased by it when given $(X_i)_{i\in[T]}$. We may assume
without loss of generality that the strategy never purchases the same
value twice. For $2\leq l\leq k$, let $T_l$ denote the set of access-$l$
templates in $B$.

\begin{claim} \label{claim:templates}
If $3\le l\le k$ is such that $\mathbb{E}|T_{l-1}|\le k^{l-1}b^{2}(\frac{T}{n})^{l-3}$ then $\mathbb{E}|T_{l}|\le k^{l}b^{2}(\frac{T}{n})^{l-2}$.
\end{claim}
\begin{claimproof}
For $0\leq s\leq T$, let $B_s$ be the set of distinct values purchased
during the first $s$ steps, and let $T_l(s)$ be the set of access-$l$
templates in $B_s$. Thus, $B_T=B$ and $T_l(T)=T_l$. Put
\[
    \Delta_{l,s}
    =
    |T_l(s)\setminus T_l(s-1)|.
\]
Let $\mathcal F_{s-1}$ denote the history up to step $s-1$, including
the internal randomness of the strategy.

Conditional on $\mathcal F_{s-1}$, the random variable $X_s$ is uniform
on $[n]$. If a template becomes an access-$l$ template at step $s$, then
before step $s$ it was already an access-$(l-1)$ template. For each
template in $T_{l-1}(s-1)$, there are at most $k$ values of $X_s$ that
can cause this template to enter $T_l(s)$. Therefore,
\[
    \mathbb E\bigl(\Delta_{l,s}\mid\mathcal F_{s-1}\bigr)
    \leq
    \frac{k|T_{l-1}(s-1)|}{n}.
\]
Summing over all steps and using
$T_{l-1}(s-1)\subseteq T_{l-1}$, we obtain
\[
\begin{aligned}
    \mathbb E|T_l|
    &=
    \sum_{s=1}^{T}\mathbb E\Delta_{l,s} \\
    &\leq
    \frac{k}{n}\sum_{s=1}^{T}
    \mathbb E|T_{l-1}(s-1)| \\
    &\leq
    \frac{kT}{n}\mathbb E|T_{l-1}|.
\end{aligned}
\]
Consequently,
\[
    \mathbb E|T_l|
    \leq
    k\mathbb E|T_{l-1}|\left(\frac{T}{n}\right)
    \leq
    k^lb^{2}\left(\frac{T}{n}\right)^{l-2}.
\]
\end{claimproof}

Notice now that deterministically, $|T_{2}|\le\binom{b}{2}\binom{k}{2} < k^{2}b^{2}$. Therefore, by iterating Claim \ref{claim:templates} exactly $k-2$ times, we obtain
\[ \mathbb{E}|T_{k}|\le k^{k}b^{2}\left(\frac{T}{n}\right)^{k-2}. \]
Since $b = o((n/t)^{\frac{k-2}{2}})$ and $T = 2t$, substituting
$b^2 = o((n/T)^{k-2})$ yields $\mathbb{E}|T_k| = o(1)$. Thus, by
Markov's inequality, \whp $|T_{k}|=0$, which implies the
0-statement. This completes the proof of Theorem \ref{thm:main}.

\bibliographystyle{custom-plainurl} % Custom numbers
\bibliography{mybib}

\end{document}